\documentclass[11pt]{article}
\usepackage[utf8]{inputenc}
\usepackage[english]{babel}

\usepackage[english]{babel}
\usepackage{amsfonts}
\usepackage{amstext}
\usepackage[usenames]{color}
\usepackage{graphicx}
\usepackage{amsmath}
\usepackage{amssymb}
\usepackage{amsthm}

\usepackage{euscript}
\usepackage{enumerate}

\numberwithin{equation}{section}
  
\graphicspath{{Graphics/}}
\newtheorem{lemma}{Lemma}
\newtheorem{theorem}{Theorem}
\newtheorem{prop}{Proposition}
\newtheorem{rem}{Remark}

\renewcommand{\Im}{\operatorname{Im}}

\renewcommand{\Re}{\operatorname{Re}}

\begin{document}
\title{Long-time asymptotics for the integrable nonlocal nonlinear Schr\"odinger equation}
\author{Ya. Rybalko$^{\dag,\ddag}$ and D. Shepelsky$^{\dag,\ddag}$\\
 \small\em {}$^\dag$ V.Karazin Kharkiv National University\\
 \small\em {}$^\ddag$ B.Verkin Institute for Low Temperature Physics and Engineering}

\date{}

\maketitle

\begin{abstract}
We study the initial value problem for 
the integrable nonlocal nonlinear Schr\"odinger (NNLS) equation
\[
iq_{t}(x,t)+q_{xx}(x,t)+2\sigma q^{2}(x,t)\bar{q}(-x,t)=0
\]
with decaying (as $x\to\pm\infty$) boundary conditions. The main aim is to describe 
 the long-time behavior  of the solution of this  problem. To do this, we adapt 
the nonlinear steepest-decent method \cite{DZ}  to 
the study of the Riemann-Hilbert problem associated with the  NNLS equation.
Our main result is that, in contrast to the local NLS equation, where 
the main asymptotic term (in the solitonless case) decays to $0$ as $O(t^{-1/2})$
along any ray $x/t=const$, the power decay rate in the case of the NNLS depends,
in general, on $x/t$, and can be expressed in terms of the spectral functions associated with the initial data.
\end{abstract}

\section{Introduction}
\label{intr}
We consider the  Cauchy problem for a nonlocal integrable variant 
of the nonlinear Schr\"odinger equation
\begin{subequations}\label{i1}
\begin{align}\label{i1-a}
&iq_{t}(x,t)+q_{xx}(x,t)+2\sigma q^{2}(x,t)\bar{q}(-x,t)=0,\,\, \sigma=\pm 1, 
	\quad x\in(-\infty,\infty),\  t>0,\\
\label{i1-b}
&q(x,0)=q_{0}(x),
\end{align}
\end{subequations} 
with the initial data $q_0(x)$ rapidly decaying to $0$ as $|x|\to\infty$
(multidimensional versions of the nonlocal nonlinear Schr\"odinger equation
are discussed in \cite{F16}). Throughout the paper,
 $\bar{q}$ denotes the complex conjugate of $q$. 

Being written in the form
\begin{equation}
\label{i2}
iq_{t}(x,t)+q_{xx}(x,t)+V(x,t)q(x,t)=0,
\end{equation}
where $V(x,t)=2\sigma q(x,t)\bar{q}(-x,t)$, equation (\ref{i1-a}) can be viewed as 
the Schr\"odinger equation with a 
``PT symmetric'' \cite{BB} potential $V(x,t)$: $V(x,t)=\overline{V}(-x,t)$.

Equation (\ref{i1-a}) has been introduced by M. Ablowitz and Z. Musslimani in \cite{AMP}
(see also \cite{G,GS,LX15,MZ,SMMC14,V}  and the references therein). 
Since the NNLS equation is PT symmetric it is invariant under the joint transformations $x\to-x$, $t\to-t$ and complex conjugation and therefore related to cutting edge research area of modern
physics \cite{BB, KYZ}. Particularly, this equation is gauge-equivalent to the unconventional system
of coupled Landau-Lifshitz (CLL) equations and therefore can be useful in the physics of nanomagnetic artificial materials \cite{GA}.
In \cite{AMN} the authors presented the Inverse Scattering Transform (IST) method to 
the study of the Cauchy problem (\ref{i1}),
based on a variant of the Riemann-Hilbert approach,
in the case of decaying initial data.
The case of particular nonzero boundary conditions in (\ref{i1-b}) 
is considered in \cite{ALM16}.
Nonlocal versions of some other integrable equations are addressed in \cite{AM16-2}
(with decaying initial data)
and \cite{AFLM17} (with nonzero boundary conditions).

The IST is known as a ``nonlinear analog'' of the Fourier transform. It allows getting 
the solution $q(x,t)$ of the original nonlinear problem by solving a series of linear problems. More precisely, starting from the given initial data,  the direct transform
is based on the analysis of the associated scattering problem, giving rise to the scattering data.
Then, the inverse transform can be characterized in terms of an 
associated  $2\times2$ matrix Riemann-Hilbert (RH), whose solution, being appropriately
evaluated in the complex plane of the spectral parameter, gives  the solution of the original initial value problem. 
In this way, the analysis of the latter reduces to the analysis of the solution of the RH problem.

Particularly, for the analysis of the long time behavior of the solution 
of  the initial value problem
for a nonlinear integrable equation, the so-called 
nonlinear version of the steepest-decent method 
has proved to be extremely efficient. 
Being developed by many people (see, e.g., \cite{DIZ} for a nice presentation
of this development), the method has finally been put into a rigorous shape by Deift and Zhou in \cite{DZ}. The idea is to perform a series of explicit and invertible transformations of the original RH problem,
in order to finally reduce it (having good control over eventual approximations) to
a form that can be solved iteratively. Particularly, 
for initial value problems with decaying initial data, 
the original RH problem reduces first to that having the jump matrix decaying fast 
(as $t\to\infty$) to the identity
matrix on the whole contour but some points on it; thus it is the parts of the contour
in small vicinities of these points that determine the main long time asymptotic terms,
which can be obtained in an explicit form after rescaling the RH problem
\cite{DZ, DIZ}.

In this paper we adapt this idea of the asymptotic analysis to the case of equation 
(\ref{i1-a}). The main ingredients  for our construction of the original RH problem
--- the Jost solutions of the associated
Lax pair equations ---
are discussed in \cite{AMN}. In Section 2 we present
the formalism of the IST method in the form of a multiplicative RH problem
suitable for the asymptotic 
 (as $t\to\infty$) analysis. This analysis is presented in Section 3,
where the main result of the paper is formulated.
Our analysis holds under some assumptions on the 
behavior of the spectral functions associated with the initial data. The validity 
of these assumptions in terms of the initial data 
(particularly, for box-shaped initial data) is discussed in Section 4.

\section{Inverse scattering transform and the Riemann-Hilbert problem}\label{ist}

The general method of AKNS \cite{AS} was applied to the NNLS equation (\ref{i1-a})
in \cite{AMN}. Here we slightly reformulate the IST and present it in the form
convenient for the consequent asymptotic analysis.

The  NNLS is the compatibility condition of two linear equations (Lax pair) for a $2\times2$-valued function $\Phi(x,t,k)$
\begin{subequations}\label{ist1}
\begin{align}\label{ist1-a}
&\Phi_{x}+ik\sigma_{3}\Phi=U\Phi\\
\label{ist1-b}
&\Phi_{t}+2ik^{2}\sigma_{3}\Phi=V\Phi,
\end{align}
\end{subequations} 
where 
\begin{equation}
U=\begin{pmatrix}
0& q(x,t)\\
-\sigma \bar{q}(-x,t)& 0\\
\end{pmatrix}
\end{equation}
for the AKNS scattering problem (\ref{ist1-a})
and
\begin{equation}
V=\begin{pmatrix}
A& B\\
C& -A\\
\end{pmatrix}
\end{equation}
with $A=i\sigma q(x,t)\bar{q}(-x,t)$, $B=2kq(x,t)+iq_{x}(x,t)$, and 
$C=-2k\sigma \bar{q}(-x,t)+i\sigma (\bar{q}(-x,t))_{x}$,
for the time evolution equation (\ref{ist1-b}).

Assuming that $q(x,t)$ satisfies (\ref{i1-a})
and that $q(\cdot, t)\in L^1(-\infty,\infty)$ for all $t\ge 0$,
 define the (matrix-valued) Jost solutions $\Phi_{j}(x,t,k)$, $j=1,2$, of (\ref{ist1}) as follows: 
$\Phi_j(x,t,k)=\Psi_j(x,t,k)e^{(-ikx-2ik^{2}t)\sigma_{3}}$, where $\Psi_j(x,t,k)$ solve the  Volterra integral equations associated with (\ref{ist1}):
\begin{subequations}\label{int}
\begin{align}\label{int-a}
&\Psi_1(x,t,k)=I+\int_{-\infty}^{x}e^{ik(y-x)\sigma_3}U(y,t)\Psi_1(y,t,k)e^{-ik(y-x)\sigma_3}\,dy,\quad
k\in(\mathbb{C}^+,\mathbb{C}^-), \\
\label{int-b}
&\Psi_2(x,t,k)=I+\int_{\infty}^{x}e^{ik(y-x)\sigma_3}U(y,t)\Psi_2(y,t,k)e^{-ik(y-x)\sigma_3}\,dy,\quad
k\in(\mathbb{C}^-,\mathbb{C}^+).
\end{align}
\end{subequations}
Here $I$ is the $2\times2$ identity matrix, 
$\mathbb{C}^{\pm}=\left\{k\in\mathbb{C}\,|\pm\Im k>0\right\}$, and 
 $k\in(\mathbb{C}^+,\mathbb{C}^-)$ means that the first and the second column 
of a matrix-valued function can be analytically continued into the upper and the lower half plane respectively. 
Moreover, from (\ref{int}) it follows that the columns of 
$\Psi_j(\cdot,\cdot, k)$ are continuous up to the boundary ($\mathbb R$)
of the corresponding half-plane and
$\Psi_j(\cdot,\cdot, k)\to I$ as $k\to \infty$, $k\in(\mathbb{C}^+,\mathbb{C}^-)$.
Since the matrix $U$  is traceless,  it follows that 
$\det\Phi_j(x,t,k)\equiv1$ for all $x$, $t$, and $k$.

Define the scattering matrix $S(k)$, $k\in\mathbb{R}$ ($\det S(k)\equiv 1$)
as the matrix relating the solutions of the system (\ref{ist1}) $\Phi_j(x,t,k)$ (for all $x$ and $t$) for $k\in\mathbb R$:
\begin{equation}
\label{ist2}
\Phi_1(x,t,k)=\Phi_2(x,t,k)S(k),\qquad k\in\mathbb{R}.
\end{equation}
Notice \cite{AMN} that the symmetry 
\begin{equation}\label{phi-sym}
\Lambda\overline{\Phi_1(-x,t,-\bar{k})}\Lambda^{-1}=\Phi_2(x,t,k),
\end{equation}
where 
$\Lambda=\bigl(
\begin{smallmatrix}
0 & \sigma\\1 & 0
\end{smallmatrix})$, 
relates the Jost solutions normalized at different infinities (w.r.t $x$),
which implies that $S(k)$ can be written in the form
\begin{equation}
\label{ist3}
S(k)=\begin{pmatrix}
a_{1}(k)& -\sigma\overline{b(-\bar{k})}\\
b(k)& a_2(k)\\
\end{pmatrix}
,\,\,k\in\mathbb{R}
\end{equation}
with some $b(k)$, $a_1(k)$, and $a_2(k)$, where the last two scattering functions satisfy 
the symmetry relations
\begin{equation}
\label{a-sym}
\overline{a_1(-\bar{k})} = a_1(k), \qquad \overline{a_2(-\bar{k})} = a_2(k).
\end{equation}
Notice that in contrast with case of  the conventional (local) NLS equation, where $\overline{a_1(\bar{k})}=a_2(k)$,
the scattering functions $a_1(k)$ and $a_2(k)$ 
for the NNLS equation are not directly related; instead, they
satisfy the symmetry relations (\ref{a-sym}),
which are not generally satisfied in the case of the NLS equation.

The scattering matrix $S(k)$ is uniquely determined by the initial data $q(x,0)$
through $S(k)=\Phi_2^{-1}(x,0,k)\Phi_1(x,0,k) = e^{ikx}\Psi_2^{-1}(x,0,k)\Psi_1(x,0,k)e^{-ikx}$, where $\Psi_j(x,0,k)$ are the solutions of 
(\ref{int}) with $q$ replaced by $q(x,0)$. Indeed, denoting 
$\psi_{1}(x,t)={(\Psi_1)}_{11}(x,0,k)$, 
$\psi_{2}(x,t)={(\Psi_1)}_{12}(x,0,k)$,
$\psi_{3}(x,t)={(\Psi_1)}_{21}(x,0,k)$,
$\psi_{4}(x,t)={(\Psi_1)}_{22}(x,0,k)$,
we have the system of equations for $\psi_{1}, \psi_{3}$
and $\psi_{2}, \psi_{4}$:
\begin{gather}
\label{ist3.6}
\left\{
\begin{array}{lcl}
\psi_{1}(x,k)=1+\int_{-\infty}^{x}q_{0}(y)\psi_3(y,k)\,dy,\\
\psi_3(x,k)=-\sigma\int_{-\infty}^{x}e^{2ik(x-y)}\overline{q_{0}(-y)}\psi_1(y,k)\,dy\\
\end{array}
\right.
\end{gather}
and
\begin{gather}
\label{ist3.8}
\left\{
\begin{array}{lcl}
\psi_2(x,k)=\int_{-\infty}^{x}e^{2ik(y-x)}q_{0}(y)\psi_4(y,k)\,dy,\\
\psi_4(x,k)=1-\sigma\int_{-\infty}^{x}\overline{q_{0}(-y)}\psi_2(y,k)\,dy.\\
\end{array}
\right.
\end{gather}
Then the scattering functions are determined in terms of the
solutions of these equations: 
\begin{equation}
\label{ist3.7}
a_{1}(k)=\lim\limits_{x\rightarrow+\infty}\psi_1(x,k),\quad 
b(k)=\lim\limits_{x\rightarrow+\infty}e^{-2ikx}\psi_3(x,k),
\end{equation}
and
\begin{equation}
\label{ist3.9}
a_{2}(k)=\lim\limits_{x\rightarrow+\infty}\psi_4(x,k).
\end{equation}

Notice that in terms of $\Psi_j$, 
the scattering relation (\ref{ist2}) reads
\begin{equation}
\label{ist3.5}
\Psi_{1}(x,t,k)=\Psi_2(x,t,k)e^{-(ikx+2ik^2t)\sigma_3}S(k)e^{(ikx+2ik^2t)\sigma_3}.
\end{equation}

 Let us summarize the properties of the spectral functions.
\begin{enumerate}
	\item 
	$a_{1}(k)$ is analytic in  $k\in\mathbb{C}^{+}$;
and continuous in 
$\overline{\mathbb{C}^{+}}$;
 $a_{2}(k)$ is analytic in $k\in\mathbb{C}^{-}$
and continuous in 
$\overline{\mathbb{C}^{-}}$. 
\item
 $a_{j}(k)=1+{O}\left(\frac{1}{k}\right)$, $j=1,2$ and 
$b(k)={O}\left(\frac{1}{k}\right)$ as $k\rightarrow\infty$ (the latter holds
for $k\in{\mathbb R}$).
\item
$\overline{a_{1}(-\bar{k})}=a_1(k)$,  
$k\in\overline{\mathbb{C}^{+}}$; \qquad
$\overline{a_{2}(-\bar{k})}=a_2(k)$,  
$k\in\overline{\mathbb{C}^{-}}$.
\item
 $a_{1}(k)a_{2}(k)+\sigma b(k)\overline{b(-
\bar{k})}=1$, $k\in{\mathbb R}$ (follows from $\det S(k)=1$).
\end{enumerate}

Now, in analogy with the case of the NLS equation,
let us define the matrix-valued function $M$ as follows:
\begin{equation}
\label{M}
M(x,t,k)=\begin{cases}
\left(\frac{\Psi_1^{(1)}(x,t,k)}{a_{1}(k)},\Psi_2^{(2)}(x,t,k)\right), & \Im k>0, \\
\left(\Psi_2^{(1)}(x,t,k),\frac{\Psi_1^{(2)}(x,t,k)}{a_{2}(k)}\right), & \Im k<0,
\end{cases}
\end{equation}
where $\Psi_i^{(j)}(x,t,k)$ is the $j$-th column of the matrix $\Psi_i(x,t,k)$. Then the scattering relation (\ref{ist3.5})
can be rewritten as the jump condition for $M(\cdot,\cdot,k)$
across $k\in {\mathbb R}$:
\begin{equation}
\label{ist4}
M_{+}(x,t,k)=M_{-}(x,t,k)J(x,t,k),\qquad k\in {\mathbb R},
\end{equation}
where
$M_\pm$ denoted the limiting value of $M$ as $k$ approaches $\mathbb R$ from 
${\mathbb C}^\pm$, and 
\begin{equation}
\label{ist4.3}
J(x,t,k)=
\begin{pmatrix}
1+\sigma r_{1}(k)r_{2}(k)& \sigma r_{2}(k)e^{-2ikx-4ik^2t}\\
r_1(k)e^{2ikx+4ik^2t}& 1
\end{pmatrix},\qquad k\in {\mathbb R}
\end{equation}
with 
\begin{equation}\label{r-12}
r_1(k):=\frac{b(k)}{a_1(k)}\qquad \text{and}\ 
r_2(k):=\frac{\overline{b(-{k})}}{a_2(k)}.
\end{equation}
Additionally, we have 
\begin{equation}
\label{norm}
M(x,t,k)\rightarrow I \quad \text{as}\  k\rightarrow\infty,
\end{equation}

Throughout the paper we assume that $a_1(k)$ and $a_2(k)$ have no zeros 
in $\overline{\mathbb{C}^{+}}$ and $\overline{\mathbb{C}^{-}}$ respectively.
Then the relations (\ref{ist4})--(\ref{norm}) 
can be viewed as the Riemann-Hilbert (RH) problem:
given $r_{1}(k)$ and $r_{2}(k)$ for
$k\in {\mathbb R}$, find a piecewise (relative to ${\mathbb R}$)
analytic, $2\times 2$ function $M$ satisfying conditions
(\ref{ist4}) and (\ref{norm}) 
(if $a_{1}(k)$ and/or $a_{2}(k)$ have zeros, then 
$M(\cdot,\cdot,k)$ can have poles at  these zeros, and
the formulation
of the RH problem is to be complemented by the associated residue 
conditions, see, e.g., \cite{FT}).

Assuming that the RH problem has a solution $M(x,t,k)$ 
for all $x$ and $t$, the solution 
$q(x,t)$ of the Cauchy problem (\ref{i1})
can be expressed in terms of the $(12)$ entry of $M(x,t,k)$ as follows \cite{AS,FT}:
\begin{equation}
\label{ist4.35}
q(x,t)=2i\lim_{k\rightarrow\infty}k\, M_{12}(x,t,k).
\end{equation}

\begin{rem}
In view of (\ref{a-sym}),
 $r_1(k)$ and $r_2(k)$ for $k\in  {\mathbb R}$ are related as follows:
\begin{equation}\label{r1r2}
\overline{r_2(-k)} = \frac{b(k)}{\overline{a_2(-k)}} = r_1(k)\frac{a_1(k)}{\overline{a_2(-k)}}
	=r_1(k)\frac{a_1(k)}{a_2(k)}.
\end{equation}
Similarly, 
\[
\overline{r_1(-k)} = r_2(k)\frac{a_2(k)}{a_1(k)}
\]
and thus
\begin{equation}\label{r-sym}
r_1(-k) r_2(-k)  = \overline{r_1(k)}\;  \overline{r_2(k)}, \quad k\in {\mathbb R}.
\end{equation}
Moreover, it follows from the determinant property 4 above that 
\begin{equation}\label{r-a}
1+\sigma r_1(k) r_2(k)=\frac{1}{a_1(k)a_2(k)},\qquad k\in {\mathbb R}.
\end{equation}
From (\ref{r1r2}) and (\ref{r-a}) it follows that given $r_1(k)$
and $r_2(k)$ for $k\in {\mathbb R}$, $a_1(k)$
and $a_2(k)$  are uniquely determined for $k\in {\mathbb R}$
(and thus, by analyticity, in ${\mathbb C}^+$ and  ${\mathbb C}^-$, respectively).
\end{rem}

\begin{rem}
If  $a_{1}(k)$ and $a_{2}(k)$ have no zeros 
in respectively $\overline{\mathbb{C}^{+}}$ and $\overline{\mathbb{C}^{-}}$,
then, by the Cauchy theorem and the Plemelj-Sokhotskii formulas for $\log a_1(k)$
and $\log a_2(k)$ (see, e.g., \cite{FT}), we have
\[
\log a_1(k) = \frac{1}{\pi i}\int_{-\infty}^\infty \frac{\log a_1(\zeta)}{\zeta-k}d\zeta,
\qquad k \in{\mathbb R},
\]
\[
\log a_2(k) = -\frac{1}{\pi i}\int_{-\infty}^\infty \frac{\log a_2(\zeta)}{\zeta-k}d\zeta,
\qquad k \in{\mathbb R}
\]
and thus
\begin{equation}\label{a1a2}
\frac{a_1(k)}{a_2(k)} = \exp\left\{\frac{1}{\pi i}\int_{-\infty}^{\infty}\frac{\ln a_1(\zeta)a_2(\zeta)}{\zeta-k}\,d\zeta \right\}, \qquad k \in{\mathbb R}.
\end{equation}
Moreover, in view of (\ref{r1r2}) we have
\begin{equation}\label{r1r2a}
\frac{\overline{r_2(-k)}}{r_1(k)} = \frac{r_2(k)}{\overline{r_1(-k)}}
	= \exp\left\{\frac{1}{\pi i}\int_{-\infty}^{\infty}\frac{\ln a_1(\zeta)a_2(\zeta)}{\zeta-k}\,d\zeta \right\}, \qquad k \in{\mathbb R}.
\end{equation}
\end{rem}

\begin{lemma}\label{lemma-sym}
The solution $M$ of the Riemann--Hilbert problem (\ref{ist4}), (\ref{ist4.3}), (\ref{norm})
satisfies the following symmetry condition (cf. (\ref{phi-sym})):
\begin{equation}
M(x,t,k)=\begin{cases}
\Lambda \overline{M(-x,t,-\bar k)} \Lambda^{-1} 
	\begin{pmatrix}
		\frac{1}{a_1(k)} & 0 \\ 0 & a_1(k)
	\end{pmatrix}, & k\in{\mathbb C}^+ \\
	\Lambda \overline{M(-x,t,-\bar k)} \Lambda^{-1}
	\begin{pmatrix}
		a_2(k) & 0 \\ 0 & \frac{1}{a_2(k)}
	\end{pmatrix}, & k\in{\mathbb C}^-
\end{cases}
\label{M-sym}
\end{equation}
\end{lemma}
\begin{proof}
It follows from (\ref{r-sym}) and (\ref{r-a})  that the jump matrix (\ref{ist4.3})
in (\ref{ist4}) satisfies the symmetry condition
\[
\Lambda \overline{J(-x,t,-k)} \Lambda^{-1} = \begin{pmatrix}
		a_2(k) & 0 \\ 0 & \frac{1}{a_2(k)}\end{pmatrix} J(x,t,k)\begin{pmatrix}
		a_1(k) & 0 \\ 0 & \frac{1}{a_1(k)}
	\end{pmatrix}.
\]
Then, taking into account the uniqueness of the solution of the Riemann--Hilbert problem
(which is due to the Liouville theorem),
the statement of the lemma follows.
\end{proof}


\begin{rem}
The symmetry (\ref{M-sym}) holds true as well 
 for the solution of the Riemann--Hilbert problem in a more general setting,
including the possibility of a finite number of zeros of  $a_{1}(k)$ and $a_{2}(k)$,
where the appropriate residue conditions are added to the formulation of the problem. 
Indeed, the proof is based on the uniqueness of the solution of the RH
problem, and the structure of the residue conditions is such that the uniqueness
is preserved (see, e.g., \cite{FT}).
 \end{rem}

\section{The long-time behavior of the nonlocal NLS}
\label{as}
The main goal of our paper is to obtain the 
long-time asymptotics of the solution to the Cauchy problem (\ref{i1}).  
Similarly to the case of the NLS equation \cite{DIZ},
we use the representation of the solution in terms of the 
solution of the Riemann--Hilbert problem, and treat it 
asymptotically by adapting
the nonlinear steepest-decent method  for  oscillatory Riemann-Hilbert problems \cite{DZ}. 
The method is based on consecutive transformations
of  the original RH problem, in order to reduce it to
an explicitly solvable problem.
The method's steps that we follow are similar to those
 in the case of the NLS equation \cite{DIZ}:
triangular factorizations of the jump matrix; ``absorbing'' the triangular factors
with good large-time behavior; reducing, after rescaling, the RH problem to that solvable
in terms of the parabolic cylinder functions; checking the approximation errors.
So, while following these steps, we will mainly point out the peculiar features 
 of the NNLS equation.

\subsection{Jump factorizations}

Introduce $\xi:=\frac{x}{4t}$ and 
\[
\theta(k,\xi)=4k\xi+2k^2.
\]
The jump matrix (\ref{ist4.3}) allows two triangular factorizations:
\begin{eqnarray}
\label{as2.2}
J(x,t,k)&=&
\begin{pmatrix}
1& 0\\
\frac{r_1(k)}{1+\sigma r_1(k)r_2(k)}e^{2it\theta}& 1\\
\end{pmatrix}
\begin{pmatrix}
1+\sigma r_1(k)r_2(k)& 0\\
0& \frac{1}{1+\sigma r_1(k)r_2(k)}\\
\end{pmatrix}
\begin{pmatrix}
1& \frac{\sigma r_2(k)}{1+\sigma r_1(k)r_2(k)}e^{-2it\theta}\\
0& 1\\
\end{pmatrix}
\\
\label{as2.3}
&=&
\begin{pmatrix}
1& \sigma r_2(k)e^{-2it\theta}\\
0& 1\\
\end{pmatrix}
\begin{pmatrix}
1& 0\\
r_1(k)e^{2it\theta}& 1\\
\end{pmatrix}.
\end{eqnarray}

Following \cite{DIZ}, the factorizations (\ref{as2.2}) and (\ref{as2.3})
are to be used in such a way that the (oscillating) jump matrix on $\mathbb R$ for a modified RH problem reduces (see the RH problem for $\hat M$ below) to the identity matrix whereas the arising jumps outside $\mathbb R$
are exponentially small as $t\to\infty$. According to the ``signature table'' for $\theta$,
i.e., the distribution of signs of $\Im \theta$ in the $k$-plane,
the factorization (\ref{as2.3}) is appropriate for $k>-\xi$ whereas
the factorization (\ref{as2.2}) is appropriated, after getting rid of the diagonal
factor, for $k<-\xi$. In turn, the diagonal factor in (\ref{as2.2})
can be handled using the solution of the scalar
 RH problem: find a scalar function $\delta(k,\xi)$ ($\xi$ is a parameter) analytic in 
${\mathbb C}\setminus (-\infty, -\xi]$  such that
\begin{equation}
\label{as2.4}
\left\{
\begin{array}{lcl}
\delta_+(k,\xi)=\delta_-(k,\xi)(1+\sigma r_1(k)r_2(k)),\,\,\,k\in (-\infty,-\xi);\\
\delta(k,\xi)\rightarrow 1,\,\,\,k\rightarrow\infty.
\end{array}
\right.
\end{equation}
Having such $\delta(k,\xi)$ constructed, we can define 
$\tilde{M}(x,t,k):=M(x,t,k)\delta^{-\sigma_3}(k,\xi)$,
which satisfies the jump condition 
\begin{equation}
\label{as3}
\tilde{M}_+(x,t,k)=\tilde{M}_-(x,t,k)\tilde{J}(x,t,k),\,\,k\in\mathbb{R}
\end{equation}
with 
\begin{equation}
\label{as4}
\tilde{J}(x,t,k)=
\left\{
\begin{array}{lcl}
\begin{pmatrix}
1& 0\\
\frac{r_1(k)\delta_-^{-2}(k,\xi)}{1+\sigma r_1(k)r_2(k)}e^{2it\theta}& 1\\
\end{pmatrix}
\begin{pmatrix}
1& \frac{\sigma r_2(k)\delta_+^{2}(k,\xi)}{1+\sigma r_1(k)r_2(k)}e^{-2it\theta}\\
0& 1\\
\end{pmatrix},\,k<-\xi
\\ \\
\begin{pmatrix}
1& \sigma r_2(k)\delta^2(k,\xi)e^{-2it\theta}\\
0& 1\\
\end{pmatrix}
\begin{pmatrix}
1& 0\\
r_1(k)\delta^{-2}(k,\xi)e^{2it\theta}& 1\\
\end{pmatrix},\,k>-\xi
\end{array}
\right.
\end{equation}
and the original normalization condition $\tilde{M}(x,t,k)\to I$ as $k\to\infty$.

A function $\delta(k,\xi)$ satisfying (\ref{as2.4})
is given by
\begin{equation}
\label{delta}
\delta(k,\xi)\equiv\exp\left\{\frac{1}{2\pi i}\int_{-\infty}^{-\xi}\frac{\ln(1+\sigma r_1(\zeta)r_2(\zeta))}{\zeta-k}\,d\zeta\right\}.
\end{equation}
At this point, we emphasize a major difference from the case of the NLS equation:
the function  $1+\sigma r_1(k)r_2(k)$ is, in general, not real-valued in the case of the NNLS. 
Consequently, $\delta(k,\xi)$ in (\ref{delta}) can be unbounded at $k=-\xi$.
Indeed, it can be written as 
\begin{equation}\label{delta-singular}
\delta(k,\xi)=
(\xi+k)^{i\nu(-\xi)}e^{\chi(k)},
\end{equation}
where
\begin{equation}\label{as5.5}
\chi(k):=-\frac{1}{2\pi i}\int_{-\infty}^{-\xi}\ln(k-\zeta)d_{\zeta}\ln(1+\sigma r_1(\zeta)r_2(\zeta))
\end{equation}
and 
\begin{equation}\label{nu}
\nu(-\xi):=-\frac{1}{2\pi}\ln(1+\sigma r_1(-\xi)r_2(-\xi)),
\end{equation}
so that 
\[
\Im \nu(-\xi)=-\frac{1}{2\pi}\int_{-\infty}^{-\xi}d\arg(1+\sigma r_1(\zeta)r_2(\zeta)).
\]

Assuming that 
\begin{equation}\label{arg-ass}
|\Im \nu(k)|<\frac{1}{2}\qquad \text{for all}\ \ k\in\mathbb{R},
\end{equation}
it follows that $\ln(1+\sigma r_1(k)r_2(k))$ is single-valued and that the singularity
of $\delta(k,\xi)$ at $k=-\xi$ is square integrable. Moreover, the singularity
of the solution $\hat M$ of the modified RH problem, see below, is also square integrable,
which will allow us to  follow closely the derivation of the main asymptotic 
term (as $t\to\infty$) in the case of bounded $\delta(k,\xi)$,
see, e.g., \cite{Len15}. But most importantly, 
assumption (\ref{arg-ass}) ensures the solvability, for sufficiently large $t$,
of  the integral equation for $\mu$, see (\ref{cauchy-estimate}) and (\ref{mu-eq}) below,
in terms of which $\hat M$ (and thus $q(x,t)$)
can be expressed, see Section 3.3.

Finally, notice the symmetries of $\nu$ and $\chi$, involved in the construction of $\delta$,
which are specific for the case of the NNLS.

\begin{lemma}
\label{lemma1}
Assuming (\ref{arg-ass}),
the following equalities hold for $\xi\in\mathbb R$:
 \begin{equation}\label{nu-sym}
\nu(-\xi) = \overline{\nu(\xi)} 
\end{equation}
and 
\begin{equation}\label{chi-sym}
\overline{\chi(-\xi)}+\chi(\xi)=\frac{1}{2\pi i}\int_{-\infty}^{\infty}\frac{\ln a_1(\zeta)a_2(\zeta)}{\zeta-\xi}\,d\zeta.
\end{equation}
\end{lemma}
\noindent\textit{Proof}: follows from (\ref{a-sym}) and (\ref{r-a}).

\subsection{RH problem transformations}

Notice that, as in the case of the NLS equation, the reflection coefficients
 $r_j(k)$, $j=1,2$, are defined, in general, for $k\in\mathbb{R}$ only; however,
one can approximate $r_j(k)$ and  $\frac{r_j(k)}{1+\sigma r_1(k)r_2(k)}$ by some rational functions with well-controlled errors (see \cite{DIZ}). 
Alternatively, if we assume that the initial data $q_0(x)$ decay to $0$ as $|x|\to\infty$
exponentially fast, then $r_j(k)$ turn out to be analytic in a band containing
$k\in\mathbb{R}$ and thus there is no need to use rational approximations
in order to be able to perform the  transformation below (the contours $\hat\gamma$
can be drawn inside the band).

Thus, keeping 
the same notations for the analytic approximations of  $r_j(k)$ and 
$\frac{r_j(k)}{1+\sigma r_1(k)r_2(k)}$ if needed, we define $\hat{M}(x,t,k)$ as follows 
(see Figure \ref{mod1})
\begin{equation}
\hat{M}(x,t,k)=
\begin{cases}
\tilde{M}(x,t,k), & k\in\hat\Omega_0,\\
\tilde{M}(x,t,k)
\begin{pmatrix}
1& \frac{-\sigma r_2(k)\delta^{2}(k,\xi)}{1+\sigma r_1(k)r_2(k)}e^{-2it\theta}\\
0& 1\\
\end{pmatrix}, & k\in\hat\Omega_1, \\
\tilde{M}(x,t,k)
\begin{pmatrix}
1& 0\\
-r_1(k)\delta^{-2}(k,\xi)e^{2it\theta}& 1\\
\end{pmatrix}, & 
k\in\hat\Omega_2,
\\
\tilde{M}(x,t,k)
\begin{pmatrix}
1& \sigma r_2(k)\delta^2(k,\xi)e^{-2it\theta}\\
0& 1\\
\end{pmatrix}, & k\in\hat\Omega_3,
\\
\tilde{M}(x,t,k)
\begin{pmatrix}
1& 0\\
\frac{r_1(k)\delta^{-2}(k,\xi)}{1+\sigma r_1(k)r_2(k)}e^{2it\theta}& 1\\
\end{pmatrix}, & k\in\hat\Omega_4.
\end{cases}
\end{equation}
\begin{figure}[h]
\begin{minipage}[h]{0.45\linewidth}
\centering{\includegraphics[width=1\linewidth]{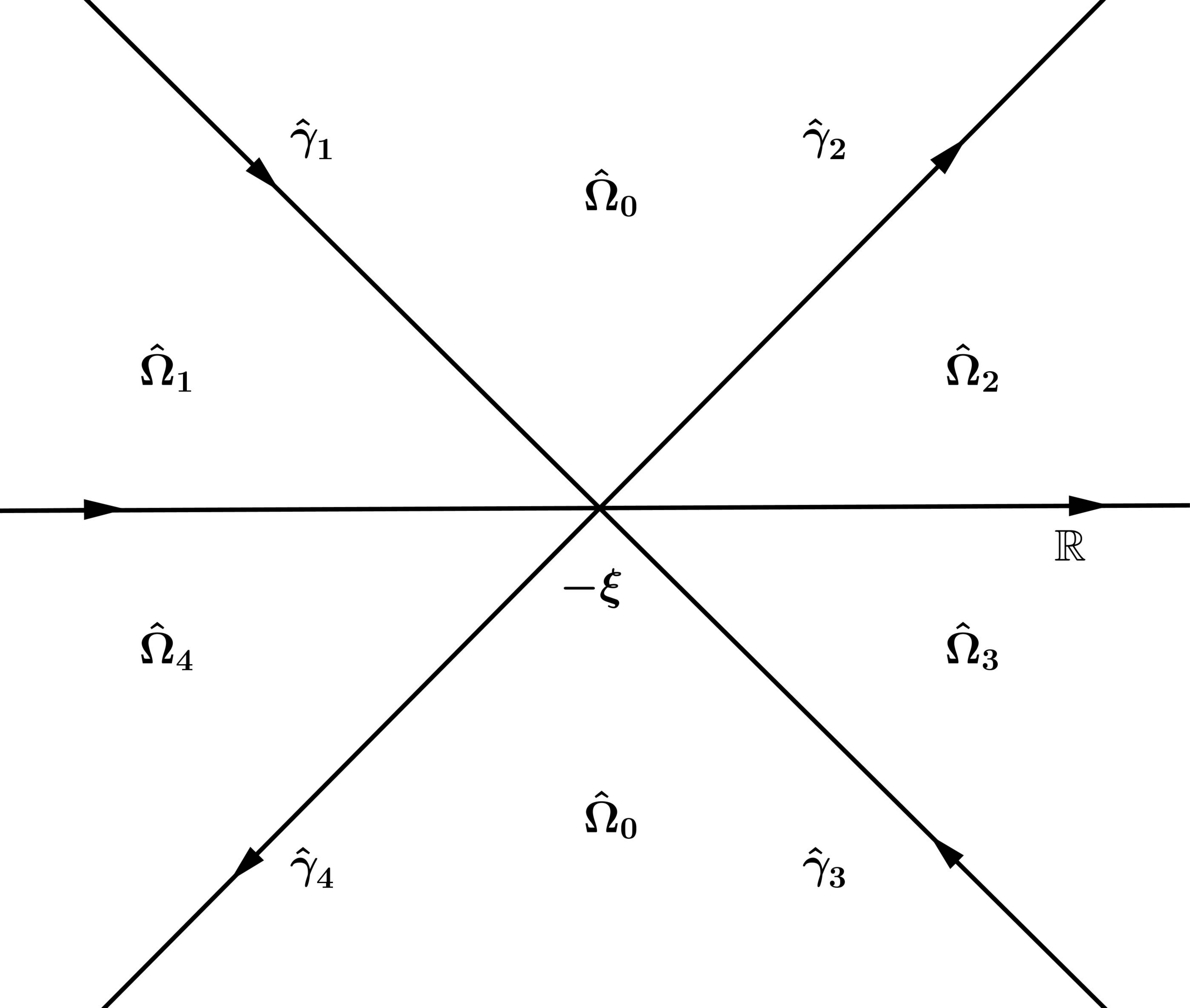}}
\caption{}
\textit{$\hat\Gamma=\hat\gamma_1\cup...\cup\hat\gamma_4$ }
\label{mod1}
\end{minipage}
\hfill
\begin{minipage}[h]{0.45\linewidth}
\centering{\includegraphics[width=1\linewidth]{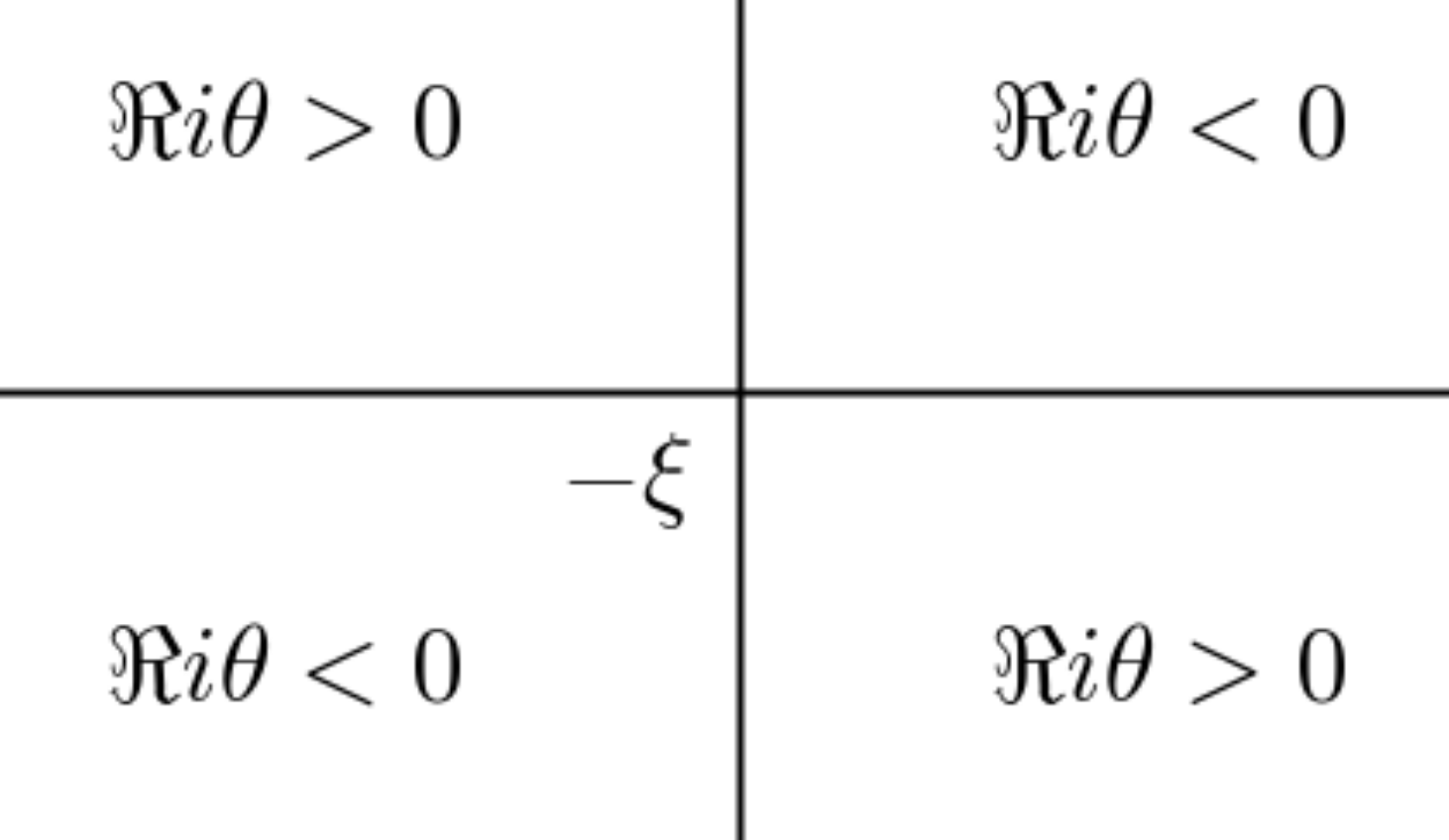}}
\caption{}
\textit{Signature table}
\label{signtable}
\end{minipage}
\end{figure}
Then $\hat{M}$ satisfies the following RH problem with jumps across $\hat\Gamma$:
\begin{equation}
\label{as5}
\begin{array}{lcl}
\hat{M}_+(x,t,k)=\hat{M}_-(x,t,k)\hat{J}(x,t,k),\,\,k\in\hat\Gamma,\\
\hat{M}(x,t,k)\rightarrow I,\,\,k\rightarrow\infty,
\end{array}
\end{equation}
where 
\begin{equation}\label{j-hat}
\hat{J}(x,t,k)=
\begin{cases}
\begin{pmatrix}
1& \frac{\sigma r_2(k)\delta^{2}(k,\xi)}{1+\sigma r_1(k)r_2(k)}e^{-2it\theta}\\
0& 1\\
\end{pmatrix}, & k\in\hat\gamma_1,
\\
\begin{pmatrix}
1& 0\\
r_1(k)\delta^{-2}(k,\xi)e^{2it\theta}& 1\\
\end{pmatrix}, & k\in\hat\gamma_2,
\\
\begin{pmatrix}
1& -\sigma r_2(k)\delta^2(k,\xi)e^{-2it\theta}\\
0& 1\\
\end{pmatrix}, & k\in\hat\gamma_3,
\\
\begin{pmatrix}
1& 0\\
\frac{-r_1(k)\delta^{-2}(k,\xi)}{1+\sigma r_1(k)r_2(k)}e^{2it\theta}& 1\\
\end{pmatrix}, & k\in\hat\gamma_4.
\end{cases}
\end{equation}

Under assumption (\ref{arg-ass}), $\hat M$ can have a singularity at $k= -\xi$
 weaker
 than $(k+\xi)^{-1/2}$ and thus the standard arguments show that the 
solution of the RH problem (\ref{as5}) is unique, if exists.

The signature table for $\Im \theta$ clearly shows  that the jump matrix $\hat{J}(x,t,k)$ for the deformed problem (\ref{as5}) converges rapidly, as $t\to\infty$ to the identity matrix
for all $k\in{\hat\Gamma\setminus \{-\xi\}}$,
and the convergence is  uniform outside any fixed neighborhood of $-\xi$.

\subsection{Reduction to a model RH problem}

According to the general scheme of the asymptotic analysis of RH problems
with a stationary phase point of $\theta$, 
near which the uniform decay to $I$ 
of the jump matrix is violated  \cite{DZ,DIZ}, 
the solution of the original RH problem
is compared with that of a problem modified in a neighborhood of this point
using a rescaled spectral parameter. Since the rescaling is dictated by $\theta(k)$,
and  $\theta$ in our case is the same as in the case of the NLS equation, 
the rescaled spectral parameter $z$ is  introduced in the same way,  by
\begin{equation}
k=\frac{z}{\sqrt{8t}}-\xi.
\end{equation}
Then
$$e^{2it\theta}=e^{i\frac{z^{2}}{2}-4it\xi^2}.
$$

Near $k=-\xi$, the functions involved in the jump matrix (\ref{j-hat})
can be approximated  as follows:
$$r_j(k(z))\approx r_j(-\xi),\quad j=1,2,
 $$
$$
\delta(k(z),\xi)\approx\left(\frac{z}{\sqrt{8t}}\right)^{i\nu(-\xi)}e^{\chi(-\xi)}.
$$
This suggests the following strategy (we basically follow \cite{Len15}) for the asymptotic analysis, as $t\to\infty$, of
the solution of the RH problem  (\ref{as5}) and, consequently, $q(x,t)$: 
\begin{enumerate}[(i)]
	\item 
	introduce the local parametrix, i.e., 
the solution $m^{\Gamma}(\xi,z)$ of the RH problem  in the $z$ plane
relative to 
$\Gamma=\gamma_1\cup...\cup\gamma_4$ (see Figure \ref{mod2}):
\begin{equation}
\label{as6}
\left\{
\begin{array}{lcl}
m^{\Gamma}_+(\xi,z)=m^{\Gamma}_-(\xi,z)j^{\Gamma}(\xi,z),\,\,z\in\Gamma\\
m^{\Gamma}(\xi,z)\rightarrow I,\,\,z\rightarrow\infty,
\end{array}
\right.
\end{equation}
where 
\begin{equation}
j^{\Gamma}(\xi,z)=
\begin{cases}
\begin{pmatrix}
1& \frac{\sigma r_2(-\xi)}{1+\sigma r_1(-\xi)r_2(-\xi)}e^{-i\frac{z^2}{2}}z^{2i\nu(-\xi)}\\
0& 1\\
\end{pmatrix}, & k\in\gamma_1,
\\
\begin{pmatrix}
1& 0\\
r_1(-\xi)e^{i\frac{z^2}{2}}z^{-2i\nu(-\xi)}& 1\\
\end{pmatrix}, & k\in\gamma_2,
\\
\begin{pmatrix}
1& -\sigma r_2(-\xi)e^{-i\frac{z^2}{2}}z^{2i\nu(-\xi)}\\
0& 1\\
\end{pmatrix}, & k\in\gamma_3,
\\
\begin{pmatrix}
1& 0\\
\frac{-r_1(-\xi)}{1+\sigma r_1(-\xi)r_2(-\xi)}e^{i\frac{z^2}{2}}z^{-2i\nu(-\xi)}& 1\\
\end{pmatrix}, & k\in\gamma_4.
\end{cases}
\end{equation}
\item
using $m^{\Gamma}(\xi,z)$, introduce $\tilde m_0(x,t,k)$ for $k$ near $-\xi$:
\begin{equation}\label{m0-tilde}
\tilde m_0(x,t,k)=\Delta(\xi,t)m^{\Gamma}(\xi,z(k)) \Delta^{-1}(\xi,t),
\end{equation}
where 
\begin{equation}\label{Delta}
\Delta(\xi,t) = e^{(2 i t \xi^2 + \chi(-\xi))\sigma_3}(8t)^{-\frac{i\nu(-\xi)}{2}\sigma_3};
\end{equation}
\item
using $\tilde m_0(x,t,k)$, introduce $\hat m(x,t,k)$:
\begin{equation}\label{m-hat-def}
\hat m(x,t,k) = \begin{cases}
\hat M(x,t,k) \tilde m_0^{-1}(x,t,k), & |k+\xi|<\varepsilon \\
\hat M(x,t,k), & \text{otherwise}
\end{cases}
\end{equation}
(with some small $\varepsilon>0$); then $\hat m(x,t,k)$ is the solution of the 
following RH problem:
\begin{equation}\label{m-hat-RHP}
\begin{cases}
\hat m_+(x,t,k) = \hat m_-(x,t,k)\hat v(x,t,k), & k\in \hat\Gamma_1\\
\hat m(x,t,k)\to I, & k\to \infty
\end{cases}
\end{equation}
where $\hat\Gamma_1=\hat\Gamma\cup\{k:|k+\xi|=\varepsilon\}$ (the circle $\{k:|k+\xi|=\varepsilon\}$ is counterclockwise oriented) and 
\begin{equation}\label{v-hat}
\hat v(x,t,k) = \begin{cases}
\tilde m_{0-}(x,t,k) \hat J(x,t,k)\tilde m_{0+}^{-1} (x,t,k), & k\in \hat\Gamma, |k+\xi|<\varepsilon \\
\tilde m_{0}^{-1} (x,t,k), & |k+\xi|=\varepsilon \\
\hat J(x,t,k), & \text{otherwise}
\end{cases}
\end{equation}
Let us note, that (\ref{ist4.35}) implies
\begin{equation}\label{sol-1}
q(x,t)=2i\lim_{k\rightarrow\infty}k\, \hat{m}_{12}(x,t,k).
\end{equation}
\item
estimating $\hat v - I$, establish estimates for the large-$t$ behavior of 
$\hat m(x,t,k)$ and thus $\hat M(x,t,k)$ in view of establishing 
the large-$t$ behavior of $q(x,t)=2i(\hat M_1)_{12}(x,t)$ (cf. (\ref{ist4.35})), 
where $(\hat M_1)_{12}(x,t)$ comes from the large-$k$ development of $\hat M$:
$\hat M(x,t,k) = I+\frac{\hat M_1(x,t)}{k}+\dots$.
\end{enumerate}

\begin{figure}[h]
\begin{minipage}[h]{0.99\linewidth}
\centering{\includegraphics[width=0.5\linewidth]{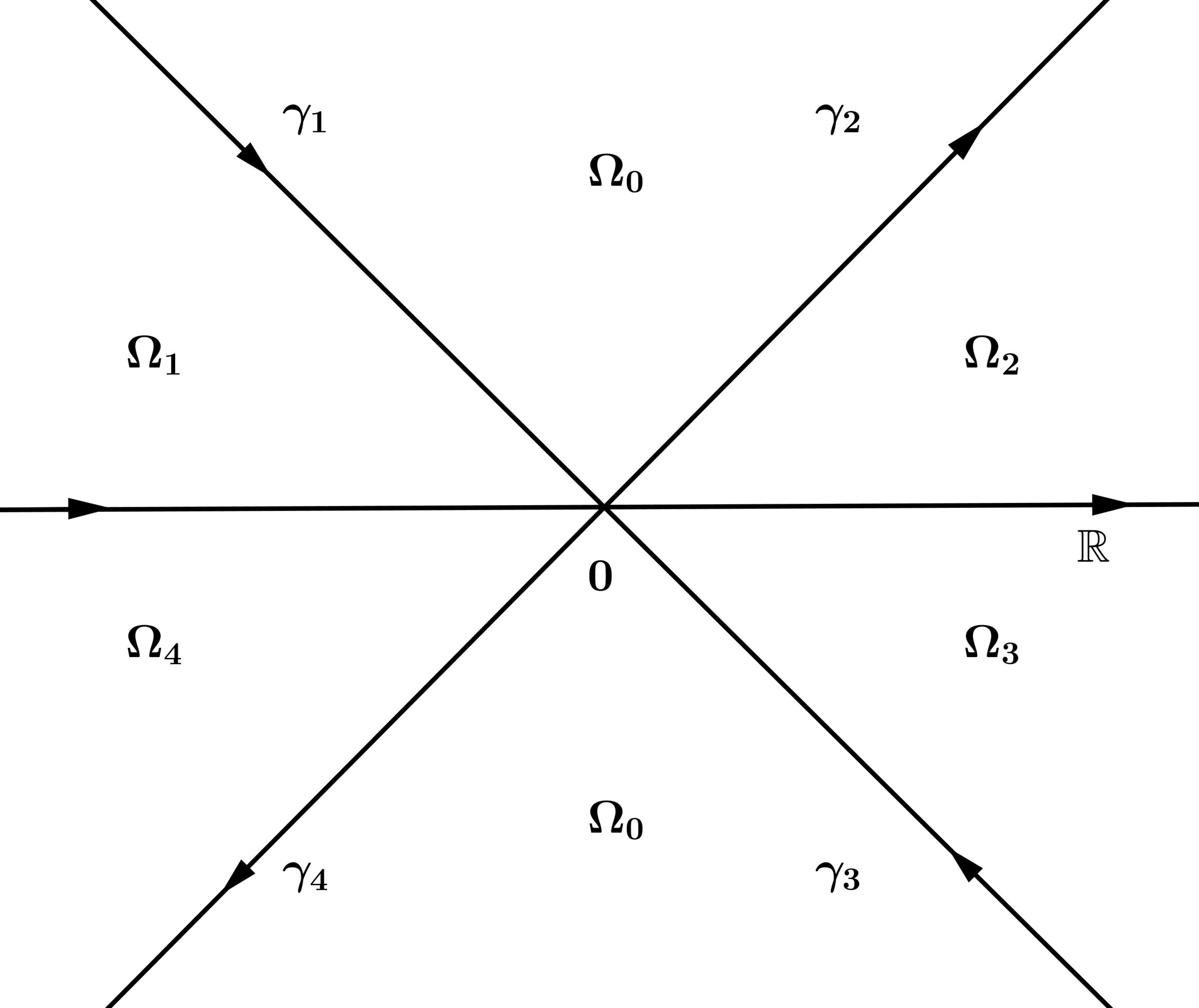}}
\caption{}
\textit{$\Gamma=\gamma_1\cup...\cup\gamma_4$ }
\label{mod2}
\end{minipage}
\end{figure}

Let us discuss the implementation of this strategy.
Concerning item (i), we notice that  $m^{\Gamma}(\xi,z)$ can be constructed as follows:
\begin{equation}\label{m-g-0}
m^{\Gamma}(\xi,z) = m_0(\xi,z) D^{-1}_{j}(\xi,z),\qquad z\in\Omega_j,\,\,j=0,\ldots,4,
\end{equation}
where 
$$D_0(\xi,z)=e^{-i\frac{z^2}{4}\sigma_3}z^{i\nu(-\xi)\sigma_3},
$$
\begin{equation}
\nonumber
\begin{matrix}
D_1(\xi,z)=D_0(\xi,z)
\begin{pmatrix}
1& \frac{\sigma r_2(-\xi)}{1+\sigma r_1(-\xi)r_2(-\xi)}\\
0& 1\\
\end{pmatrix},
&&
D_2(\xi,z)=D_0(\xi,z)
\begin{pmatrix}
1& 0\\
r_1(-\xi)& 1\\
\end{pmatrix},\\
D_3(\xi,z)=D_0(\xi,z)
\begin{pmatrix}
1& -\sigma r_2(-\xi)\\
0& 1\\
\end{pmatrix},
&&
D_4(\xi,z)=D_0(\xi,z)
\begin{pmatrix}
1& 0\\
\frac{-r_1(-\xi)}{1+\sigma r_1(-\xi)r_2(-\xi)}& 1\\
\end{pmatrix},
\end{matrix}
\end{equation}
and  $m_0(\xi,z)$ is the solution of the following RH problem, relative to $\mathbb R$,
 with a \textit{constant} jump matrix:
 	\begin{equation}\label{as8}
\left\{
\begin{array}{lcl}
m_{0+}(\xi,z)=m_{0-}(\xi,z)j_0(\xi),\,\,z\in\mathbb{R},\\
m_0(\xi,z)= \left(I+O(1/z)\right)
e^{-i\frac{z^2}{4}\sigma_3}z^{i\nu(-\xi)\sigma_3},\,\,z\rightarrow\infty,
\end{array}
\right.
\end{equation}
where  
\begin{equation}\label{j0}
j_0(\xi)=
\begin{pmatrix}
1+\sigma r_1(-\xi)r_2(-\xi) &\sigma r_2(-\xi)\\
r_1(-\xi) & 1
\end{pmatrix}.
\end{equation}
The solution $m_0(\xi,z)$ of this problem can be given explicitly, in terms of the parabolic
cylinder functions, see Appendix A.

From (\ref{v-hat}) and (\ref{m-g-0}) it follows that  $\hat w(x,t,k) :=\hat{v}(x,t,k)-I\in L^{2}(\hat\Gamma_1)\cap L^{\infty}(\hat\Gamma_1)$ and thus
 we can closely follow the proof of Theorem 2.1
in \cite{Len15}, particularly, in what is related to the (large-$t$)
 estimates for $\hat w(x,t,k)$.
Here the major difference between our case and that in \cite{Len15}
is that the factor $\Delta(\xi,t)$, see (\ref{Delta}) ($D(\xi,t)$ in the notations 
in \cite{Len15}) now contains an increasing (with $t$) term 
$t^{\frac{\Im\nu(-\xi)}{2}}$ or 
$t^{\frac{-\Im\nu(-\xi)}{2}}$ (depending on the sign of $\Im\nu(-\xi)$).
Particularly, this implies that in our case, the estimates 
for the columns of the $\hat{w}(x,t,k)$ (see Claims 1-4
in \cite{Len15}) acquire the following form (recall that $A^{(j)}$ denotes the j-th column of a matrix $A$):
\[
\hat{w}^{(j)}(\xi,t,\cdot) = 
O\left(t^{-\frac{1}{2} +(-1)^j \Im \nu(-\xi)} 
e^{-a(k+\xi)^2t}\ln t\right), \quad t\to\infty,\quad k\in \hat \Gamma\cap\{|k+\xi|<\varepsilon\},\,a >0,\,j=1,2.
\]
uniformly with respect to $k$ in the given ranges. Moreover,
\[
\|\hat{w}^{(j)}(\xi,t,\cdot)\|_{L^2(\hat \Gamma_1)}
=O\left(t^{-\frac{1}{2}+(-1)^{j}\Im \nu(-\xi)}
\right),\quad t\to\infty,\,j=1,2,
\]
\[
\|\hat{w}(\xi,t,\cdot)\|_{L^\infty(\hat \Gamma_1) }
=O\left(t^{-\frac{1}{2}+|\Im \nu(-\xi)|}\ln t
\right),\quad t\to\infty,
\]
and
\[
\|\hat{w}^{(j)}(\xi,t,\cdot)\|_{L^1(\hat \Gamma)}
=O\left(t^{-1+(-1)^j\Im \nu(-\xi)}\ln t\right), \quad t\to\infty,\,j=1,2.
\]

Particularly, the above estimates imply
\begin{equation}\label{cauchy-estimate}
\|{\cal C}_{\hat{w}}\|_{{\cal B}(L^2(\hat \Gamma_1))} 
= O\left( t^{-\frac{1}{2}+|\Im \nu(-\xi)|}\ln t\right),\quad t\to\infty,
\end{equation}
where the integral operator 
${\cal C}_{\hat{w}}: L^2(\hat \Gamma_1) + L^{\infty}(\hat \Gamma_1)\to L^2(\hat \Gamma_1)$ is defined by 
${\cal C}_{\hat{w}} f = {\cal C}_-(f\hat{w})$ and 
${\cal C}$ is the Cauchy operator associated with $\hat \Gamma_1$:
\[
({\cal C}f )(z) = \frac{1}{2\pi i}\int_{\hat \Gamma_1}\frac{f(s)}{s-z} ds, 
\qquad z\in {\mathbb C}\setminus \hat \Gamma_1.
\]
Moreover,

\[
\|\mu(\xi,t,\cdot) - I\|_{L^2(\hat \Gamma_1)} = O\left( t^{-\frac{1}{2} +|\Im \nu(-\xi)|}\right),\mbox{ as }\,t\to\infty,
\]
where $\mu = I+L^2(\hat \Gamma_1)$ is the solution of the integral equation 
\begin{equation}\label{mu-eq}
\mu -  {\cal C}_{\hat w} \mu = I
\end{equation}
(or, more precisely, 
$\mu-I \in L^2(\hat \Gamma_1)$ is the solution of the integral equation
$(I- {\cal C}_{\hat w})(\mu-I)= {\cal C}_{\hat w} I$).

Finally, we use the representation for $\hat m$ 
(where $\hat{m}(x,t,k)\equiv\hat{m}(\xi,t,k)$):
\[
\hat m(\xi,t,k) = I + {\cal C}(\mu\hat w) 
	= I + \frac{1}{2\pi i}\int_{\hat \Gamma_1}\mu(\xi,t,s)\hat w(\xi,t,s)\frac{ds}{s-k},
\]
which implies that 
\begin{equation}\label{hat-m-lim}
\begin{aligned}
\lim\limits_{k\to\infty} k(\hat m(\xi,t,k)-I) 
	& = - \frac{1}{2\pi i}\int_{\hat \Gamma_1}\mu(\xi,t,k)\hat w(\xi,t,k)dk \\
	& = - \frac{1}{2\pi i}\int_{|k+\xi|=\varepsilon}\mu(\xi,t,k)(\tilde m_0^{-1}(\xi,t,k)-I)dk
		-  \frac{1}{2\pi i}\int_{\hat \Gamma}\mu(\xi,t,k)\hat w(\xi,t,k)dk.
		\end{aligned}
\end{equation}

Using the estimates above and the restriction on the $\Im\nu(-\xi)$, namely $-\frac{1}{2}<\Im\nu(-\xi)<\frac{1}{2}$ (see (\ref{nu}) and (\ref{arg-ass})), the columns of the second term in (\ref{hat-m-lim}) can be estimated, as $t\to\infty$,
as follows  (cf. \cite{Len15}):
\begin{subequations}\label{err}
\begin{equation}\label{err1}
\int_{\hat \Gamma}
\left(\mu(\xi,t,k) \hat{w}(\xi,t,k)\right)^{(1)}\,dk
=
\begin{cases}
O\left(t^{-1}\right),& \Im\nu(-\xi)>0,\\
O\left(t^{-1}\ln t\right),&\Im\nu(-\xi)=0,\\
O\left(t^{-1+2|\Im\nu(-\xi)|}\right),&\Im\nu(-\xi)<0,
\end{cases}
\end{equation}
and
\begin{equation}\label{err2}
\int_{\hat \Gamma}
\left(\mu(\xi,t,k) \hat{w}(\xi,t,k)\right)^{(2)}\,dk=
\begin{cases}
O\left(t^{-1+2|\Im\nu(-\xi)|}\right),& \Im\nu(-\xi)>0,\\
O\left(t^{-1}\ln t\right),&\Im\nu(-\xi)=0,\\
O\left(t^{-1}\right),&\Im\nu(-\xi)<0.
\end{cases}
\end{equation}
\end{subequations}
On the other hand, the main asymptotic terms comes from the calculations of the first term
using the large-$t$ asymptotics for $\tilde m_0^{-1}(\xi,t,k)$, which, in turn,
comes from the large-$z$ asymptotics for $m^\Gamma(\xi, z)$. Indeed,
using
\[
m^\Gamma (\xi, z) = I + \frac{i}{z}\begin{pmatrix}
0 & \beta(\xi) \\ -\gamma(\xi) & 0
\end{pmatrix} + O(z^{-2}), \quad z\to\infty
\]
(where $\beta(\xi)$ and  $\gamma(\xi)$ are given in (\ref{as17}) and (\ref{as18}), see Appendix A)
and (\ref{m0-tilde}) with (\ref{Delta}), we have

\[
\tilde m_0^{-1}(\xi,t,k) = \Delta(\xi,t)(m^\Gamma)^{-1}(\xi, \sqrt{8t}(k+\xi))
	\Delta^{-1}(\xi,t) = I + \frac{B(\xi,t)}{\sqrt{8t}(k+\xi)} + \tilde{r}_1(\xi,t),
\]

where 
\begin{equation}\label{B}
B(\xi,t)=\begin{pmatrix}
	0 & -i\beta(\xi)e^{4it \xi^2 + 2\chi(-\xi)}(8t)^{-i\nu(-\xi)} \\
	i\gamma(\xi)e^{-4it \xi^2 - 2\chi(-\xi)}(8t)^{i\nu(-\xi)} & 0
\end{pmatrix}.
\end{equation}
and the columns of the remainder are
\begin{equation}\label{rem-r1}
\tilde r_1^{(j)}(\xi,t)=O\left(t^{-1+(-1)^j\Im \nu(-\xi)}\right),\quad t\to\infty, j=1,2.
\end{equation}

Therefore, for the first $2\times2$ matrix in the right-hand side of  (\ref{hat-m-lim}) we have

\begin{equation}\label{est}
\begin{aligned}
&\int_{|k+\xi|=\varepsilon}\mu(\xi,t,k)\left(\tilde m_0^{-1}(\xi,t,k)
-I\right)\,dk 
= \int_{|k+\xi|=\varepsilon}\left(\tilde m_0^{-1}(\xi,t,k)
-I\right)\,dk\\
& + \int_{|k+\xi|=\varepsilon}
\left(\mu(\xi,t,k)-I\right)\left(\tilde m_0^{-1}(\xi,t,k)-I\right)\,dk
=2\pi i\tilde{B}(\xi,t)+\tilde r_1(\xi,t)+\tilde r_2(\xi,t),
\end{aligned}
\end{equation}
where (recall that $\{k:|k+\xi|=\varepsilon\}$ is  oriented counterclockwise)
\begin{equation}\label{B-tilde}
\tilde{B}(\xi,t)=\begin{pmatrix}
0 & -i\beta(\xi)e^{4it \xi^2 + 2\chi(-\xi)}(8t)^{-\frac{1}{2}-i\nu(-\xi)} \\
i\gamma(\xi)e^{-4it \xi^2 - 2\chi(-\xi)}(8t)^{-\frac{1}{2}+i\nu(-\xi)} &
0\end{pmatrix},
\end{equation}
and the remainder $\tilde r_2(\xi,t)\equiv\|\mu-I\|_{L^2(|k+\xi|=\varepsilon)}\cdot O\left(\tilde{B}(\xi,t)\right)$ can be estimated, as $t\to\infty$, as follows:
\begin{subequations}\label{rem-r2}
\begin{equation}
\tilde r_2^{(1)}(\xi,t)=
\begin{cases}
O\left(t^{-1}\right), & \Im\nu(-\xi)\geq 0,\\
O\left(t^{-1+2|\Im\nu(-\xi)|}\right),&\Im\nu(-\xi)<0,
\end{cases}
\end{equation}
and
\begin{equation}
\tilde r_2^{(2)}(\xi,t)=
\begin{cases}
O\left(t^{-1+2|\Im\nu(-\xi)|}\right), & \Im\nu(-\xi)\geq 0,\\
O\left(t^{-1}\right),&\Im\nu(-\xi)<0.
\end{cases}
\end{equation}
\end{subequations}
Finally, taking into account (\ref{err}), (\ref{rem-r1}), (\ref{est}) and (\ref{rem-r2}), equation (\ref{hat-m-lim}) takes the form
\begin{equation}\label{M-as-lim-2}
\lim\limits_{k\to\infty} k(\hat{m}(\xi,t,k)-I) =
\tilde{B}(\xi,t)+\tilde r_3(\xi,t),
\end{equation}
where $\tilde r_3(\xi,t)\equiv \tilde r_1(\xi,t)+\tilde r_2(\xi,t)-\frac{1}{2\pi i}\int_{\hat \Gamma}\mu(\xi,t,k) w(\xi,t,k)\,dk$ can be estimated similarly as in (\ref{err}), and we arrive at

\begin{theorem}
\label{th1}
Consider the Cauchy problem (\ref{i1}) with 
the initial data $q_0(x)\in L^1(\mathbb{R})$. Assume that the spectral functions
associated with $q_0(x)$ via (\ref{ist3.6})--(\ref{ist3.9}) are such that:
\begin{enumerate}[(i)]
	\item 
	$a_1(k)$ and $a_2(k)$ have no zeros in $\overline{\mathbb{C}^{+}}$ and $\overline{\mathbb{C}^{-}}$ respectively;
	\item
	$\int_{-\infty}^{a}d\arg(1+\sigma r_1(\zeta)r_2(\zeta))\in(-\pi,\pi)$
	for all $a\in\mathbb{R}$, where 
$r_1(k)=\frac{b(k)}{a_1(k)}$, $r_2(k)=\frac{\overline{b(-{k})}}{a_2(k)}$.
\end{enumerate}
Assuming that the solution $q(x,t)$ of (\ref{i1}) exists and $q(x,t)\in L^{1}(\mathbb{R})$ for all $t>0$, its long-time asymptotics  is as follows:
\begin{equation}
\label{as19}
q(x,t)=t^{-\frac{1}{2}+\Im\nu(-\xi)}p(-\xi)
\exp\left\{4it\xi^2-i\Re\nu(-\xi)\ln t\right\}
+R(\xi,t),\quad t\rightarrow\infty,
\end{equation}
uniformly for $x\in\mathbb{R}$ chosen from compact
subsets, where $\xi=\frac{x}{4t}$,
\[
\nu(-\xi)=-\frac{1}{2\pi}\ln(1+\sigma r_1(-\xi)r_2(-\xi)),
\]
\[
p(-\xi)=\dfrac{\sqrt{\pi}
\exp\left\{-\frac{\pi}{2}\nu(-\xi)+\frac{\pi i}{4}+2\chi(-\xi)-3i\nu(-\xi)\ln 2\right\}}{r_1(-\xi)\Gamma(-i\nu(-\xi))},
\]
where $\Gamma(\cdot)$ is Euler's Gamma function and
\[
\chi(-\xi)=-\frac{1}{2\pi i}\int_{-\infty}^{-\xi}\ln(-\xi-\zeta)d_{\zeta}\ln(1+\sigma r_1(\zeta)r_2(\zeta)),
\]
and the remainder is estimated as follows:
\begin{equation*}
R(\xi,t)=
\begin{cases}
O\left(t^{-1+2|\Im\nu(-\xi)|}\right),& \Im\nu(-\xi)>0,\\
O\left(t^{-1}\ln t\right),&\Im\nu(-\xi)=0,\\
O\left(t^{-1}\right),&\Im\nu(-\xi)<0.
\end{cases}
\end{equation*}
\end{theorem}

\begin{rem}
In contrast to the local NLS equation, where 
the main asymptotic term decays to $0$ as $O(t^{-1/2})$
along any ray $\xi=const$, the power decay rate in the case of the NNLS depends,
in general, on $\xi$, through the imaginary part of $\nu(-\xi)$.
\end{rem}

\begin{rem}
The symmetry (\ref{M-sym}) implies that the solution $q(x,t)$ can also be obtained by 
$$
q(x,t)=-2i\lim_{k\to\infty}k\overline{M_{21}(-x,t,k)}=
-2i\lim_{k\to\infty}k\overline{\hat{m}_{21}(-x,t,k)},
$$
It follows from (\ref{as17}) and (\ref{as18}) 
(taking into account (\ref{nu-sym}), (\ref{chi-sym}) and (\ref{r1r2a})) that 
\begin{equation}
\label{as7}
\gamma(\xi)=-\sigma e^{2(\overline{\chi(\xi)}+\chi(-\xi))}\overline{\beta(-\xi)},
\end{equation} 
and, therefore, the asymptotic formula for $q(x,t)$ can be also obtained via the (21) entry in (\ref{M-as-lim-2}),
which is, due to (\ref{nu-sym}) and (\ref{as7}), consistent with the asymptotic formula (\ref{as19}).
\end{rem}

\begin{rem}
In the case of  even initial data, the solution of the Cauchy problem (\ref{i1}) 
is also even and thus it satisfies the classical (local) NLS equation. 
In this case, $ r_1(k)=\bar{r}_2(k)=:r(k)$ and thus 
$$\nu(-\xi)=-\frac{1}{2\pi}\ln(1+\sigma|r(-\xi)|^2),$$
$$\chi(k)=-\frac{1}{2\pi i}\int_{-\infty}^{-\xi}\ln(k-\zeta)d_{\zeta}
	\ln(1+\sigma |r(\zeta)|^2),$$
and the asymptotic formula takes the  well-known form \cite{DIZ}
\begin{equation}
\label{as20}
q(x,t)=t^{-\frac{1}{2}}p(-\xi)
\exp\left\{4it\xi^2-i\nu(-\xi)\ln t\right\}+O(t^{-1}\ln t),\qquad 
t\rightarrow\infty,
\end{equation}
where
$$|p(-\xi)|^2=\frac{\sigma}{4\pi}\ln(1+\sigma|r(-\xi)|^2)$$
and
$$\arg p(-\xi)=-3\nu(-\xi)\ln2+\frac{\pi}{4}+\arg\Gamma(i\nu)-\arg r(-\xi)-2i\chi(-\xi).$$
\end{rem}
\begin{rem}
\label{rem1}
A sufficient condition for  assumptions (i) and (ii) in 
Theorem \ref{th1} to hold can be written in terms of  $\|q_0(x)\|_{L^1(\mathbb{R})}$ :
\begin{equation}
\int_{-\infty}^{\infty}|q_0(x)|\,dx<0.817.
\end{equation}
Indeed, by (\ref{ist3.6}) and (\ref{ist3.7}), $a_1(k)$ can be estimated, following \cite{AS},
by using the Neumann series for the solution $\psi_1$ of the Volterra integral equation
\begin{equation}
\psi_1(x,k)=1-\sigma\int_{-\infty}^{x}\overline{q_{0}}(-y)\int_y^x e^{2ik(y-z)}q_0(z)\,dz\,\psi_1(y,k)\,dy.
\end{equation}
Introducing 
$Q_0(x)=\int_{-\infty}^{x}|q_0(x)|\,dx$, $R_0(x)=\int_{-\infty}^{x}|q_{\,0}(-x)|\,dx$, 
the estimates of the Neumann series
 give 
\[
|\psi_1(x,k)-1|\leq 
	Q_0(x)R_0(x)+\frac{1}{(2!)^2}Q_0^2(x)R_0^2(x)+\frac{1}{(3!)^2}Q_0^3(x)R_0^3(x)+\ldots=
I_0(2\sqrt{Q_0(x)R_0(x)}) -1,
\]
where $I_0(\cdot)$ is the modified Bessel function.
Since $R_0(\infty)=Q_0(\infty)$, it follows that if $I_0(2Q_0(\infty))<2$ 
(which holds true, in particular,  if $Q_0(\infty)<0.817$), then 
$a_1(k)$ for $\Im k\in\overline{\mathbb{C}^{+}}$ can be represented as $a_1(k)=1+F_1(k)$ with 
$|F_1(k)|<1$. Similarly for $a_2(k)$ for $\Im k\in\overline{\mathbb{C}^{-}}$.
Consequently, (i) $a_j(k)$, $j=1,2$ have no zeros in the respective half-planes;
(ii) $\arg a_j(k)\in \left(-\frac{\pi}{2},\frac{\pi}{2}\right)$, which, in turn, 
implies item (ii) in Theorem \ref{th1} taking into account (\ref{r-a}).
\end{rem}
\begin{rem}
If the Riemann--Hilbert problem (\ref{ist4})--(\ref{norm})
has a solution for all $x$ and $t$, then the standard ``dressing'' 
arguments (see, e.g., \cite{FT})
imply that the solution $q(x,t)$ of (\ref{i1}) decaying as $x\to\pm\infty$ does exist.
On the other hand, if the $L^1$ norm of the initial data $q_0(x)$
is small enough, then it follows from (\ref{ist3.6}) and (\ref{ist3.7})
that the $L^\infty$ norms of $r_1$ and $r_2$ are small as well, which in turn
implies (see, e.g., \cite{DZ03}) that the 
integral equation associated with the Riemann--Hilbert problem (\ref{ist4})--(\ref{norm})
(and thus the RHP itself)
is solvable for all $x$ and $t$.

With this respect, we notice  that the pure soliton solutions can be used
(see \cite{G})
to demonstrate that the smallness of the $H^1$ norm of the initial data $q_0(x)$
does not guarantee the global existence of the solution of (\ref{i1}).
\end{rem}

\section{Single box initial values}
\label{sb}
In Section 3 we have obtained the long-time asymptotic formula for the solution of the Cauchy  problem (\ref{i1}) under assumption that the spectral functions associated with the initial data satisfy  conditions  (i) and (ii)
 in Theorem \ref{th1}. In this section we verify this conditions in the case of the so-called single box initial data
(cf. \cite{AMN}, section 11),
for which the spectral functions can easily be calculated explicitly.

Consider the initial data $q_0(x)$ of the form 
\begin{equation}\label{sb0.5}
q_0(x) = \begin{cases}
0, & x<0, \\
H, & 0<x<L,\\
0, & x>L,
\end{cases}
\end{equation}
where $H\in\mathbb{C}$ and $L>0$. 
Using  equations (\ref{ist3.6}), (\ref{ist3.8}) and the definitions  (\ref{ist3.7}) and 
(\ref{ist3.9}), the spectral functions are given by  
\begin{equation}
\begin{aligned}
a_1(k) & =1+\frac{\sigma|H|^2}{4k^2}\left(e^{2ikL}-1\right)^2, \\
a_2(k) & \equiv 1, \\
b(k) & =\frac{\sigma\,\overline{H}}{2ik}\left(1-e^{2ikL}\right).
\end{aligned}
\end{equation}
Notice that since $a_2(k)  \equiv 1$ in this case, it follows that 
$1+\sigma r_1(k)r_2(k)=a_1^{-1}(k)$ and thus
 $\int_{-\infty}^{\xi}d\arg(1+\sigma r_1(\zeta)r_2(\zeta))
=-\int_{-\infty}^{\xi}d\arg a_1(\zeta)$.

\begin{prop}\label{prop1}
Let $\sigma=1$. The following results hold:
\begin{itemize}
\item If $|H|L<1$, then $a_1(k)$ has no zeros for $k\in\overline{\mathbb{C}^{+}}$ and 
$\arg a_1(k)\in\left(-\frac{\pi}{2},\frac{\pi}{2}\right)$ for all $k\in\mathbb{R}$;
\item If $|H|L>1$, then $a_1(k)$ has at least one zero in $\mathbb{C}^{+}$,
 and there exists $k\in\mathbb{R}$ such that $\int_{-\infty}^{k}d\arg a_1(\zeta)>\frac{\pi}{2}$.
\end{itemize}
\end{prop}
\begin{proof}
First, notice that 
$a_1(0)=1-|H|^2L^2\ne 0$ in both cases. 
Now, introducing $z=z_1+iz_2$ for $z\in\mathbb{C}^{+}$, 
the existence of 
 $k\in\mathbb{C}^{+}\setminus\{0\}$ such that  $a_1(k)=0$ is equivalent to
the solvability of one of the systems
\begin{equation}
\label{sb1}
\left\{
\begin{array}{lcl}
e^{-z_2}\cos z_1-1=\mp\frac{z_2}{C}\\
e^{-z_2}\sin z_1=\pm\frac{z_1}{C}
\end{array}
\right.
\end{equation}
where $C=|H|L$ and  $z=2kL$ ($z\ne 0$).
Setting  $z_1=0$,  the system (\ref{sb1}) reduces to the equations 
$e^{-z_2}-1=\frac{z_2}{C}$ or $e^{-z_2}-1=-\frac{z_2}{C}$;
the first   equation has no solutions with  $z_2>0$ whereas
 the second one  has a solution 
with  $z_2>0$ if and only if
$C>1$. In the case $C<1$, the second equations in (\ref{sb1}) have no solutions with $z_1\neq0$ and $z_2\geq0$, and thus in this case $a_1(k)$ has no zeros 
in  $\overline{\mathbb{C}^{+}}$.

Concerning the argument of $a_1(k)$, we notice that if $|H|L<1$, then 
 $\Re a_1(k) = 1-\frac{|H|^2}{k^2}\sin^{2}kL\cos 2kL>0$ for  $k\in\mathbb{R}$
and, therefore,  $\arg a_1(k)\in\left(-\frac{\pi}{2},\frac{\pi}{2}\right)$ in this case.
 On the other hand, if $|H|L>1$, then there exists $\delta>0$ such that $\Re a_1(k)<0$ 
for all $k\in(-\delta,\delta)$ and, moreover,  
$\Im a_1(k)=-\frac{|H|^2}{k^2}\sin^2kL\sin2kL>0$ for $k\in(-\delta,0)$ and
$\Im a_1(k)<0$ for $k\in(0,\delta)$.
\end{proof}

Proposition \ref{prop1} gives an example of a set of initial data such that 
if $\|q_0(x)\|_{L^1(\mathbb{R})}<1$ for $q_0$ in this set, then  the both assumptions (i) and (ii)  in  
Theorem \ref{th1} hold, and, on the other hand, if $\|q_0(x)\|_{L^1(\mathbb{R})}>1$,
then the  both assumptions do not hold simultaneously.
The next example illustrate the situation, when 
 assumption (i) holds, but (ii) does not.
\begin{prop}
Let $\sigma=-1$. Then there exist $\varepsilon>0$ such that for 
$\frac{\pi}{2}<|H|L<\frac{\pi}{2}+\varepsilon$, $a_1(k)$ has no zeros in 
$\overline{\mathbb{C}^{+}}$, but there exists $k\in\mathbb{R}$ such that 
$\int_{-\infty}^{k}d\arg a_1(\zeta)>\frac{\pi}{2}$.
\end{prop}
\begin{proof}
 First, notice that $a_1(0)\ne 0$ because $|H|L>1$ in the considered case.
Then, in analogy with above, 
the existence of 
 $k\in\mathbb{C}^{+}\setminus\{0\}$ such that  $a_1(k)=0$ is equivalent to
\begin{equation}
\label{sb2}
\left\{
\begin{array}{lcl}
e^{-z_2}\cos z_1-1=\pm\frac{z_1}{C}\\
e^{-z_2}\sin z_1=\pm\frac{z_2}{C}
\end{array}
\right.
\end{equation}
where, as above,  $C=|H|L$ and  $z=2kL$.
Now notice that if $C$ is close to  $\frac{\pi}{2}$, then  the system $(\ref{sb2})$ has no solutions 
with $z_2=0$ and $z_1\neq0$. Further, the second equation of (\ref{sb2}) implies that there are no solutions with $z_1=0$, and the first equation implies that 
if $z_1$ is the real part of a solution, then it must satisfy $|z_1|\leq2C$. 

Due to the symmetry relation $a_1(k)=\overline{a_{1}(-\bar{k})}$,
 it is sufficient to show that (\ref{sb2}) has no solutions with $z_1>0$ and $z_2>0$. 
Let us look for the solution of (\ref{sb2}) in the form $z_2=\rho z_1$ with some $\rho>0$;
then (\ref{sb2}) reduces to 
\[
\left\{
\begin{array}{lcl}
e^{-\rho z_1}(\rho\cos z_1-\sin z_1)-\rho=0\\
e^{-\rho z_1}\sin z_1=\pm\frac{\rho z_1}{C}
\end{array}
\right.
\]
which implies, in particular, the equation for $z_1$:
\[
e^{-\rho z_1}(\rho\cos z_1-\sin z_1)=\rho.
\]
But this equation has no solutions for $0<z_1\leq\pi+2\varepsilon$ for 
 $\varepsilon>0$ small enough, which implies that $a_1(k)$ has no zeros with $|z_1|\leq2C$
with $C$ satisfying $\frac{\pi}{2}<C<\frac{\pi}{2}+\varepsilon$.

Now we notice that 
 if $C>\frac{\pi}{2}$, then $\Re a_1(k)=1+\frac{|H|^2}{k^2}\sin^{2}kL\cos 2kL<0$ for 
$k\in(\frac{\pi}{2L}-\delta,\frac{\pi}{2L}+\delta)$ with some $\delta>0$. At the same time,  
$\Im a_1(k)\equiv\frac{|H|^2}{k^2}\sin^2kL\sin2kL>0$ for 
$k\in(\frac{\pi}{2L}-\delta,\frac{\pi}{2L})$ and, therefore, 
  $|\int_{-\infty}^{k}d\arg a_1(\zeta)|>\frac{\pi}{2}$ for such $k$.
\end{proof}

\section*{Appendix A}

The model RH problem (\ref{as8}) can be solved explicitly,
similarly to the case of the local NLS,
 in terms of the parabolic cylinder functions (see \cite{I1,FIKK}). Indeed,
since the jump matrix $j_0(\xi)$ is independent of $z$, it follows that  the logarithmic derivative $\frac{d}{dz}m_0\cdot m_0^{-1}$ is an entire function. Taking into account the asymptotic condition in the RH problem (\ref{as8}) and using Liouville's theorem one concludes
(cf. \cite{Len15}) that  $m_0(\xi,z)$ satisfies the following ODE
\begin{equation}
\label{as9}
\frac{d}{dz}m_0(\xi,z)+
\begin{pmatrix}
\frac{iz}{2}& \beta(\xi)\\
\gamma(\xi)& \frac{-iz}{2}\\
\end{pmatrix}
m_0(\xi,z)=0
\end{equation}
with some $\beta(\xi)$ and $\gamma(\xi)$; moreover
(recall (\ref{m-g-0})), $\beta(\xi)=-i (m^\Gamma_1)_{12}(\xi)$, 
$\gamma(\xi)=i (m^\Gamma_1)_{21}(\xi)$,
where
$m_0(\xi,z)=(I+\frac{m^{\Gamma}_1(\xi)}{z}+\ldots)
e^{-i\frac{z^2}{4}\sigma_3}z^{i\nu(-\xi)\sigma_3}$ as $z\to\infty$.

Observe that a solution to (\ref{as9}) can be written in the form
\begin{equation}\label{m0}
m_0(\xi,z)=
\begin{pmatrix}
(m_0)_{11}(\xi,z)& \dfrac{\left(\frac{d}{dz}-i\frac{z}{2}\right)(m_0)_{22}(\xi,z)}{-\gamma(\xi)}\\
\dfrac{\left(\frac{d}{dz}+i\frac{z}{2}\right)(m_0)_{11}(\xi,z)}{-\beta(\xi)}& (m_0)_{22}(\xi,z)
\end{pmatrix}
\end{equation}
where the functions $(m_0)_{jj}(\xi,z)$, $j=1,2$ satisfy the parabolic cylinder  equations
\begin{equation}
\label{as11}
\frac{d^{2}}{dz^2}(m_0)_{11}(\xi,z)+\left(\frac{i}{2}-\beta(\xi)\gamma(\xi)+\frac{z^2}{4}\right)(m_0)_{11}(\xi,z)=0,
\end{equation}
\begin{equation}
\label{as12}
\frac{d^{2}}{dz^2}(m_0)_{22}(\xi,z)+\left(-\frac{i}{2}-\beta(\xi)\gamma(\xi)+\frac{z^2}{4}\right)(m_0)_{22}(\xi,z)=0
\end{equation}
and thus 
can be expressed in terms of the parabolic cylinder functions
$D_{i\tilde\nu(-\xi)}(\pm z e^{\frac{\pi i}{4}})$ and 
$D_{-i\tilde\nu(-\xi)}(\pm z e^{\frac{{3\pi i}}{4}})$ respectively, where 
 $\tilde\nu(-\xi):=\beta(\xi)\gamma(\xi)$.
Indeed, using the formula for derivative of the parabolic cylinder function $D_a(z)$:
\begin{equation*}
\frac{d}{dz}D_a(z)=\frac{z}{2}D_a(z)-D_{a+1}(z)
\end{equation*}
and 
comparing the asymptotic condition in (\ref{as8}) with the asymptotic formula for $D_a(z)$:

\begin{equation*}
D_a(z)=z^a e^{-\frac{z^2}{4}}\left(1-\frac{a(a-1)}{2z^2}+O(z^{-4})\right),\,\,
\arg z\in\left(-\tfrac{3\pi}{4};\tfrac{3\pi}{4}\right), \quad z\to\infty, \quad a\in\mathbb{C},
\end{equation*}
one obtains 
\begin{equation}
\label{as13}
(m_0)_{11}(\xi,z)=
\begin{cases}
e^{-\frac{3\pi}{4}\tilde\nu(-\xi)} D_{i\tilde\nu(-\xi)}\left(e^{-\frac{3\pi i}{4}}z\right),
	& \Im z>0,\\
e^{\frac{\pi}{4}\tilde\nu(-\xi)} D_{i\tilde\nu(-\xi)}\left(e^{\frac{\pi i}{4}}z\right), 
	& \Im z<0,
\end{cases}
\end{equation}
\begin{equation}
\label{as14}
(m_0)_{22}(\xi,z)=
\begin{cases}
e^{\frac{\pi}{4}\tilde\nu(-\xi)} D_{-i\tilde\nu(-\xi)}\left(e^{-\frac{\pi i}{4}}z\right),
& \Im z>0,\\
e^{-\frac{3\pi}{4}\tilde\nu(-\xi)} D_{-i\tilde\nu(-\xi)}\left(e^{\frac{3\pi i}{4}}z\right), 
& \Im z<0.
\end{cases}
\end{equation}
Now, calculating $m_{0-}^{-1}(\xi,z)m_{0+}(\xi,z)$ for $z\in\mathbb R$
from (\ref{m0}), (\ref{as13}) and (\ref{as14}), noting that 
$m_{0-}^{-1}(\xi,z)m_{0+}(\xi,z)=m_{0-}^{-1}(\xi,0)m_{0+}(\xi,0)$,
and using the equalities 
\begin{equation}
\begin{matrix}
D_{a}(0)=\dfrac{2^{\frac{a}{2}}\sqrt{\pi}}{\Gamma\left(\frac{1-a}{2}\right)},
&& D_{a}^{'}(0)=-\dfrac{2^{\frac{1+a}{2}}\sqrt{\pi}}{\Gamma\left(-\frac{a}{2}\right)},\\
\\ \Gamma\left(\frac{1}{2}-i\frac{a}{2}\right)\Gamma\left(\frac{1}{2}+i\frac{a}{2}\right)=
\dfrac{\pi}{\cosh\frac{\pi a}{2}},\,a\neq 2i\mathbb{Z}+1,
&& \Gamma\left(-i\frac{a}{2}\right)\Gamma\left(i\frac{a}{2}\right)=
\dfrac{2\pi}{a\sinh\frac{\pi a}{2}},\,a\neq 2i\mathbb{Z},
\end{matrix}
\end{equation}
where $\Gamma(\cdot)$ is Euler's Gamma function,
we find that 
\begin{eqnarray}
\nonumber
m_{0-}^{-1}(\xi,z)m_{0+}(\xi,z)\equiv m_{0-}^{-1}(\xi,0)m_{0+}(\xi,0)
=\begin{pmatrix}
\frac{e^{\frac{\pi\tilde{\nu}}{4}}2^{\frac{i\tilde\nu}{2}}\sqrt{\pi}}{\Gamma\left(\frac{1-i\tilde{\nu}}{2}\right)}
& \frac{e^{\frac{3\pi}{4}(i-\tilde\nu)}2^{\frac{1-i\tilde\nu}{2}}\sqrt{\pi}}{\gamma(\xi)\Gamma\left(\frac{i\tilde\nu}{2}\right)} \\
\frac{e^{\frac{\pi}{4}(i+\tilde\nu)}2^{\frac{1+i\tilde\nu}{2}}\sqrt{\pi}}{\beta(\xi)\Gamma\left(\frac{-i\tilde\nu}{2}\right)} 
& \frac{e^{-\frac{3\pi\tilde{\nu}}{4}}2^{-\frac{i\tilde\nu}{2}}\sqrt{\pi}}{\Gamma\left(\frac{1+i\tilde{\nu}}{2}\right)}
\end{pmatrix}^{-1}\\
\times
\begin{pmatrix}
\frac{e^{-\frac{3\pi\tilde{\nu}}{4}}2^{\frac{i\tilde\nu}{2}}\sqrt{\pi}}{\Gamma\left(\frac{1-i\tilde{\nu}}{2}\right)}
& \frac{e^{\frac{\pi}{4}(\tilde\nu-i)}2^{\frac{1-i\tilde\nu}{2}}\sqrt{\pi}}{\gamma(\xi)\Gamma\left(\frac{i\tilde\nu}{2}\right)} \\
\frac{e^{-\frac{3\pi}{4}(i+\tilde\nu)}2^{\frac{1+i\tilde\nu}{2}}\sqrt{\pi}}{\beta(\xi)\Gamma\left(\frac{-i\tilde\nu}{2}\right)} 
& \frac{e^{\frac{\pi\tilde{\nu}}{4}}2^{\frac{-i\tilde\nu}{2}}\sqrt{\pi}}{\Gamma\left(\frac{1+i\tilde{\nu}}{2}\right)}
\end{pmatrix}=
\begin{pmatrix}
e^{-2\pi\tilde\nu} & \frac{e^{-\frac{\pi\tilde\nu}{2}}e^{-\frac{\pi i}{4}}\sqrt{2\pi} }{\gamma(\xi)\Gamma(i\tilde\nu)}\\
\frac{e^{-\frac{\pi\tilde\nu}{2}}e^{-\frac{3\pi i}{4}}\sqrt{2\pi} }{\beta(\xi)\Gamma(-i\tilde\nu)} & 1
\end{pmatrix}
\end{eqnarray}
where $\tilde\nu\equiv\tilde\nu(-\xi)$.
Consequently, equating this to $j_0(\xi)$ defined in (\ref{j0}):
\begin{equation}
\label{as15}
\begin{pmatrix}
e^{-2\pi\tilde\nu} & \frac{e^{-\frac{\pi\tilde\nu}{2}}e^{-\frac{\pi i}{4}}\sqrt{2\pi} }{\gamma(\xi)\Gamma(i\tilde\nu)}\\
\frac{e^{-\frac{\pi\tilde\nu}{2}}e^{-\frac{3\pi i}{4}}\sqrt{2\pi} }{\beta(\xi)\Gamma(-i\tilde\nu)} & 1
\end{pmatrix}=\begin{pmatrix}
1+\sigma r_1(-\xi)r_2(-\xi) &\sigma r_2(-\xi)\\
r_1(-\xi) & 1
\end{pmatrix}
\end{equation}
we obtain that 
$\tilde \nu(\xi)\equiv \nu(\xi)$, where $\nu(\xi)$ is defined in (\ref{nu}), and 
 that $\beta(\xi)$ and $\gamma(\xi)$ are expressed in terms of $r_1(\xi)$
and $r_2(\xi)$ (``inverse monodromy problem'', cf. \cite{FIKK}) as follows:
\begin{equation}
\label{as17}
\beta(\xi)=\dfrac{\sqrt{2\pi}e^{-\frac{\pi}{2}\nu(-\xi)}e^{-\frac{3\pi i}{4}}}{r_1(-\xi)\Gamma(-i\nu(-\xi))},
\end{equation}
\begin{equation}
\label{as18}
\gamma(\xi)=\dfrac{\sigma\sqrt{2\pi}e^{-\frac{\pi}{2}\nu(-\xi)}e^{-\frac{\pi i}{4}}}{r_2(-\xi)\Gamma(i\nu(-\xi))}.
\end{equation}


\begin{thebibliography}{99}

\bibitem{AFLM17}
M. J. Ablowitz, B.-F. Feng, X.-D. Luo and Z. H. Musslimani,
Inverse scattering transform for the nonlocal reverse space-time sine-Gordon, sinh-Gordon
and nonlinear Schr\"odinger equations 
with nonzero boundary conditions, preprint arXiv:1703.02226.

\bibitem{ALM16}
M. J. Ablowitz, X.-D. Luo and Z. H. Musslimani,
Inverse scattering transform for the nonlocal nonlinear Schr\"odinger equation
with nonzero boundary conditions, preprint arXiv:1612.02726.
\bibitem{AMP}
M. J. Ablowitz and Z. H. Musslimani, Integrable nonlocal nonlinear Schr\"odinger equation, \textit{Phys. Rev. Lett.} \textbf{110} 064105 (2013).
\bibitem{AMN}
M. J. Ablowitz and Z. H. Musslimani, Inverse scattering transform for the integrable nonlocal nonlinear Schr\"odinger equation, \textit{Nonlinearity} \textbf{29} (2016), 915--946.
\bibitem{AM16-2}
M. J. Ablowitz and Z. H. Musslimani,
Integrable nonlocal nonlinear equations,
\textit{Stud. Appl. Math.} \textbf{139} (2017), 7--59.



\bibitem{AS}
M. J. Ablowitz and H. Segur, \textit{Solitons and Inverse Scattering Transform.} SIAM Studies
in Applied Mathematics 4. Philadelphia, PA: Society for Industrial and Applied Mathematics
(SIAM), 1981.

\bibitem{BB}
C. M. Bender and S. Boettcher, Real spectra in non-Hermitian Hamiltonians having P-T symmetry, \textit{Phys. Rev. Lett.} \textbf{80} (1998), 5243.
\bibitem{DZ}
P. A. Deift and X. Zhou, A steepest descend method for oscillatory Riemann--Hilbert prob-
lems. Asymptotics for the MKdV equation, \textit{Ann. Math.} 137, no. 2 (1993): 
295--368.
\bibitem{DZ03}
P. A. Deift and X. Zhou, 
Long-time asymptotics for solutions of the NLS equation
with initial data in a weighted Sobolev space, 
\textit{Comm. Pure  Appl. Math.} 56 (2003), 1029--1077.
\bibitem{DIZ}
P. A. Deift, A. R. Its and X. Zhou, Long-time asymptotics for integrable nonlinear wave equations.
In \textit{Important developments in Soliton Theory 1980-1990}, edited by A. S. Fokas and V. E. Zakharov, New York: Springer, 181--204, 1993.
\bibitem{FT}
L. D. Faddeev  and L. A. Takhtajan,
\textit{Hamiltonian Methods in the Theory of Solitons}.
Springer Series in Soviet Mathematics. Springer-Verlag, Berlin, 1987.
\bibitem{F16}
A. S. Fokas,
Integrable multidimensional versions of the nonlocal nonlinear Schr\"odinger equation,
\textit{Nonlinearity} 29 (2016), 319--324.
\bibitem{FIKK}
A. S. Fokas, A.R. Its, A.A. Kapaev and V. Yu. Novokshenov, 
\textit{Painleve Transcendents. The Riemann--Hilbert Approach}, AMS, 2006.
\bibitem{KYZ}
V. V. Konotop, J. Yang and D. A. Zezyulin, Nonlinear waves in PT-symmetric systems, Rev. Mod. Phys. 88,
035002 (2016).

\bibitem{GA}
T. Gadzhimuradov and A. Agalarov, Towards a gauge-equivalent magnetic structure of the nonlocal
nonlinear Schr\"odinger equation, Phys. Rev. A, 93, 062124 (2016).
\bibitem{G}
F. Genoud, Instability of an integrable nonlocal NLS, 
\textit{C. R. Math. Acad. Sci. Paris} 355 (2017), 299--303.
 arXiv:1612.03139.
\bibitem{GS}
V. S. Gerdjikov and A. Saxena,
 Complete integrability of nonlocal nonlinear Schr\"odinger equation, \textit{J. Math. Phys.} \textbf{58}, 013502 (2017).
\bibitem{I1}
A. R. Its, Asymptotic behavior of the solutions to the nonlinear Schr\"odinger equation, and isomonodromic deformations of systems of linear differential equations, \textit{Doklady Akad. Nauk SSSR} 261, no. 1 (1981), 14--18.


\bibitem{Len}
J. Lenells, Matrix Riemann-Hilbert problems with jumps across Carleson contours, preprint  	arXiv:1401.2506.
\bibitem{Len15} 
J. Lenells, The nonlinear steepest descent method for Riemann-Hilbert problems
of low regularity, 
\textit{Indiana Univ. Math.} 66 (2017), 1287--1332.
\bibitem{LX15} 
M. Li and T. Xu, 
Dark and antidark soliton interactions in the nonlocal nonlinear Schr\"odinger equation
with the self-induced parity-time-symmetric potential,
\textit{Phys. Rev. E} 91, 033202 (2015).



\bibitem{MZ}
L. Ma and Z. Zhu, Nonlocal nonlinear Schr\"odinger equation and its discrete version: Soliton solutions and gauge equivalence,
 \textit{J. Math. Phys.} 57, 083507 (2016).


\bibitem{SMMC14}
A.K. Sarma, M.-A. Miri, Z.H. Musslimani and D.N. Christodoulides,
Continuous and discrete Schr\"odinger systems with parity-time-symmetric nonlinearities,
\textit{Phys. Rev. E} 89, 052918 (2014).




\bibitem{V}
T. Valchev, On a nonlocal nonlinear Schr\"odinger equation. In \textit{Slavova, A.(ed) Mathematics in Industry, Cambridge Scholars Publishing} pp. 36--52 (2014).

\bibitem{ZM}
V. E. Zakharov and S. V. Manakov, Asymptotic behavior of nonlinear
wave systems integrated by the inverse method. \textit{Zh. Eksp. Teor. Fiz.}
\textbf{71} (1976), 203--215 (in Russian); \textit{Sov. Phys. JETP}, \textbf{44}, no. 1 (1976),
 106--112 (in English).

\end{thebibliography}
\end{document}